\documentclass{amsart}

\usepackage{amssymb}
\usepackage{amsmath}
\usepackage{amsthm}
\usepackage{amsfonts}

\usepackage{latexsym}

\usepackage[dvips]{epsfig}
\usepackage{amsbsy}
\usepackage{amsgen}
\usepackage{amscd}
\usepackage{amsopn}
\usepackage{amstext}
\usepackage{amsxtra}
\usepackage{xypic}

\newtheorem{theorem}{Theorem}[section]
\newtheorem{definition}[theorem]{Definition}
\newtheorem{example}[theorem]{Example}
\newtheorem{remark}[theorem]{Remark}
\newtheorem{lem}[theorem]{Lemma}

\newtheorem{corollary}[theorem]{Corollary}

\newtheorem{observation}[theorem]{Observation}
\newtheorem{construction}[theorem]{Construction}
\newtheorem{repl}[theorem]{Replacement}
\newtheorem{el}[theorem]{Elementary facts about 
homotopy Cartesian objects}

\begin{document}
\title[On calculus of functors in model categories]
{On calculus of functors in model categories}
\author[{A. E.} {Stanculescu}]{{Alexandru E.} {Stanculescu}}
\address{\newline Department of Mathematics and Statistics,
\newline Masarykova Univerzita, Kotl\'{a}{\v{r}}sk{\'{a}} 2,\newline
611 37 Brno, Czech Republic}
\email{stanculescu@math.muni.cz}

\begin{abstract}
We present an analysis of some constructions and 
arguments from the universe of T. G. Goodwillie's 
Calculus, in a general model theoretic setting.
\end{abstract}

\maketitle

In this paper we undertake a small analysis of T. G. 
Goodwillie's Calculus \cite{Go2},\cite{Go3} in model categories.
To the best of our knowledge, it was N. Kuhn \cite{Ku} who first
described a general model theoretic setting for the development 
of Calculus and outlined Goodwillie's main results from \cite{Go3} 
in this setting. Our work is influenced by his, but it has a somewhat
different purpose. Our analysis is mainly oriented towards making 
more conceptual some of the constructions and arguments
from the universe of Calculus, while retaining a general model 
theoretic setting. This setting differs slightly from Kuhn's. We 
hope that the analysis also reveals some of the complications that 
might arise when one wants to work in such an environment.

The first topic we address in our analysis is the construction 
of the Taylor tower. Let {\bf Cat} be the category of all small 
categories. Starting with a small subcategory {\bf J} of {\bf Cat} 
which does not contain the empty category we construct a 
sequence $\{{\bf J}(n+1)\}_{n\geq 0}$ of small subcategories of 
{\bf Cat}. ${\bf J}(n+1)$ is obtained from ${\bf J}(n)$ using a 
Grothendieck construction. When {\bf J} consists of the terminal 
object of {\bf Cat} only, ${\bf J}(n+1)$ is the set of subsets of 
the set $\{1,2,...,n+1\}$ (viewed as a preorder with inclusions 
as arrows) with the empty set removed. To each functor 
$F:\mathcal{C}\rightarrow \mathcal{D}$
between simplicial model categories and each small 
subcategory {\bf J} of {\bf Cat} which does not contain 
the empty category, we associate 
a tower of functors $\{P_{n,{\bf J}}F:\mathcal{C}
\rightarrow \mathcal{D}\}_{n\geq 0}$ which coincides 
with Goodwillie's Taylor tower $\{P_{n}F\}_{n\geq 0}$ of 
$F$ \cite{Go3} when {\bf J} consists
of the terminal object of {\bf Cat} only.

The second topic is the proof of Goodwillie's $n$-excisive
approximation theorem \cite[Theorem 1.8]{Go3}.
To each functor $F:\mathcal{C}\rightarrow \mathcal{D}$
between simplicial model categories we associate 
a tower of functors $\{\tilde{P}_{n}F:\mathcal{C}\rightarrow 
\mathcal{D}\}_{n\geq 0}$. $\tilde{P}_{n}F$ is a minor variation 
of $P_{n}F$. There is a natural map
of towers $\{\tilde{P}_{n}F\}\rightarrow \{P_{n}F\}$, and
for suitable $\mathcal{C}$ and $F$ this map is a weak 
equivalence when evaluated at cofibrant objects.
We believe that the use of the tower
$\{\tilde{P}_{n}F\}$ sheds some light into
Rezk's proof \cite {Re} of Goodwillie's theorem.

The third topic is concerned with the notions of
(strongly) homotopy (co-)Cartesian cubes. In the first 
place, there is a body of technical results, due to Goodwillie 
and others, related to homotopy Cartesian cubes. 
We introduce generalized homotopy Cartesian cubes, 
and we prove the analogue, in our context, of some 
of these results. In the second place, we give a model 
theoretic interpretation of the notions of homotopy 
co-Cartesian and strongly homotopy co-Cartesian cube.
\\

{\bf Notation.}
We denote by {\bf Cat} the category of all small 
categories and by $CAT$ that of all large categories.
For $\mathcal{C},\mathcal{D}\in CAT$,
we denote by $Fun(\mathcal{C},\mathcal{D})$
the category of functors $\mathcal{C}\rightarrow 
\mathcal{D}$ and natural transformations between them.
When $\mathcal{C}$ is small, we shall use the 
notation $\mathcal{D}^{\mathcal{C}}$ instead.
The terminal (initial, zero, respectively) object 
of a category, when it exists, is denoted by 
$\ast$ ($\emptyset$, $0$, respectively). We 
let $[1]$ be the category freely 
generated by the graph $\{0\rightarrow 1\}$, 
$pb$ the category freely generated by the graph 
$0\rightarrow 1\leftarrow 2$, $po$ the category 
freely generated by the graph $0\leftarrow 1\rightarrow 2$ 
and $\omega$ the free category generated by the graph
$\{0\rightarrow 1\rightarrow 2\rightarrow 3\rightarrow ...\}$.
If $S$ is a set, $|S|$ is its cardinal. We denote by 
{\bf S} the category of simplicial sets and by $N$ 
the nerve functor from {\bf Cat} to {\bf S}.

\section{Some Grothendieck constructions}

\subsection{Contravariant Grothendieck constructions}
Let $\mathbb{B}$ be a small category and 
$\Psi:\mathbb{B}^{op}\rightarrow {\bf Cat}$ a functor.
Recall that the (contravariant) {\bf Grothendieck 
construction} $\underset{\mathbb{B}}\int (\Psi)$ (or 
simply $\underset{\mathbb{B}}\int \Psi$)
of $\Psi$ is the category with objects $(b,x)$ where 
$b\in \mathbb{B}$ and $x\in \Psi(b)$, and arrows  
$(b,x)\rightarrow (b',y)$ are pairs
$(u,f)$ with $u:b\rightarrow b'$ in $\mathbb{B}$ 
and $f:x\rightarrow \Psi(u)(y)$ in $\Psi(b)$. 

The projection $p:\underset{\mathbb{B}}\int 
(\Psi)\rightarrow \mathbb{B}$ is a split fibration. The 
fiber category of $p$ over $b\in \mathbb{B}$ is isomorphic
to $\Psi(b)$. For each $b\in \mathbb{B}$, we denote 
by $\tau_{b}$ the natural functor from the fiber 
category of $p$ over $b$ to $\underset{\mathbb{B}}\int \Psi$.

One has the formula $$\underset{\mathbb{B}}
\int \Psi\cong \overset{b}\int \Psi(b)\times 
(\mathbb{B}\downarrow b)  \qquad (1)$$
\begin{example}
Let $J$ be a small category. We denote by 
$J_{+}$ the category obtained by adding an 
initial object $\emptyset$ to $J$. One has 
$J_{+}=\underset{[1]}\int \Psi$, where $\Psi(1)=J$ 
and $\Psi(0)=\ast$. The natural functor 
$\tau_{1}:J\rightarrow J_{+}$ is an inclusion.
Writing the coend in formula (1) as a colimit we 
have a pushout diagram
\[
   \xymatrix{
J \ar[d] \ar[r] & \ast \ar[d]\\
J\times [1] \ar[r] & \underset{[1]}\int \Psi\\
}
  \]
where the right vertical arrow sends $j\in J$ to $(j,0)$.
\end{example}
\begin{example}
Let $F:I\rightarrow J$ be a functor between small 
categories. Let $\Psi(F):pb^{op}\rightarrow {\bf Cat}$ 
be the functor 
\[
   \xymatrix{
& I=\Psi(1) \ar[dr] \ar[dl]_{F}\\
J=\Psi(0) & & \ast=\Psi(2)\\
}
  \]
One may then form $\underset{pb}\int \Psi(F)$.
We shall use this construction when $F$ is the identity
functor of a small category $J$. In this case we denote
$\underset{pb}\int \Psi(Id_{J})$ by 
$\underset{pb}\int J$. We have natural functors 
\[
   \xymatrix{
J \ar@{->} @<3pt> [r]^{\tau_{0}} 
\ar@{->} @<-3pt> [r]_{\tau_{1}}
& \underset{pb}\int J\\
  }
  \]
Writing the coend in formula (1) as a colimit 
we have a pushout diagram
\[
   \xymatrix{
J \ar[d] \ar[r] & \ast \ar[d]\\
J\times pb \ar[r] & \underset{pb}\int J\\
}
  \]
where the left vertical arrow sends $j\in J$ to $(j,2)$.
This pushout can be calculated as the pushout
\[
   \xymatrix{
J \ar[d] \ar[r] & J_{+} \ar[d] & & &  (2)\\
J\times [1] \ar[r] & \underset{pb}\int J\\
}
  \]
where the left vertical arrow sends $j\in J$ to $(j,1)$. 
One has $$(\underset{pb}\int J)_{+}\cong  J_{+}\times 
[1] \qquad (3)$$ 
\end{example}
\begin{construction}
Given a small category $J$, we construct by 
induction a sequence $\{J(n+1)\}_{n\geq 0}$ 
of small categories, as follows. $J(1)=J$ and 
$$J(n+1)=\underset{pb}\int J(n)$$
\end{construction}
The natural functors (Example 1.2) $\tau_{0},
\tau_{1}:J(n)\rightarrow \underset{pb}\int J(n)$ 
shall be denoted by $\tau_{0}^{n}$ and 
$\tau_{1}^{n}$, so that we have a sequence 
\[
   \xymatrix{
J=J(1) \ar@{->} @<3pt> [r]^{\tau_{0}^{1}} 
\ar@{->} @<-3pt> [r]_{\tau_{1}^{1}}
& J(2) \ar@{->} @<3pt> [r]^{\tau_{0}^{2}} 
\ar@{->} @<-3pt> [r]_{\tau_{1}^{2}}
& J(3) \ar@{->} @<3pt> [r]^{\tau_{0}^{3}} 
\ar@{->} @<-3pt> [r]_{\tau_{1}^{3}}
& ...\\
  }
  \]
From Example 1.2 we have 
$$J(n+1)_{+}\cong J(n)_{+}\times [1]$$
\begin{example}
$\ast(3)$ is
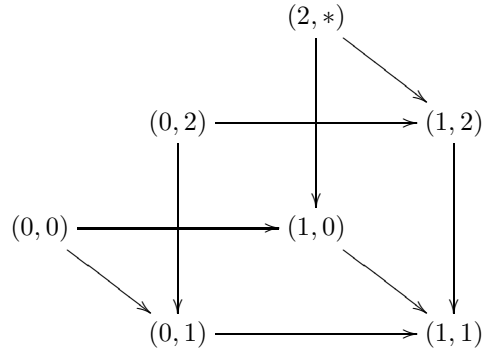
\[
\xymatrix{
& & (2,\ast) \ar[dr] \ar[dd]\\
& (0,2) \ar[rr] \ar[dd] & & (1,2) \ar[dd]\\
(0,0) \ar[rr] \ar[dr] &  & (1,0) \ar[dr]\\
&  (0,1) \ar[rr] & & (1,1)\\
 }
   \]
and this has already been observed in \cite[Example 38.2]{CS}
(and perhaps other places). In general, let $S$ be a finite set. 
We denote by $\mathcal{P}(S)$ the set of all subsets of 
$S$. $\mathcal{P}(S)$ is a preorder under inclusion. We 
let $\mathcal{P}_{0}(S)=\mathcal{P}(S)-\{\emptyset\}$ 
and $\mathcal{P}_{1}(S)=\mathcal{P}(S)-\{S\}$. For $n\geq 0$ 
we let $\underline{n}=\{1,2,...,n\}$ and $\underline{0}=\emptyset$.
One has $$\ast(n+1)\cong 
\mathcal{P}_{0}(\underline{n+1}) \quad (n\geq 0)$$
The maps $\tau_{0}^{n},\tau_{1}^{n}:
\mathcal{P}_{0}(\underline{n})\rightarrow 
\mathcal{P}_{0}(\underline{n+1})$ are 
$\tau_{0}^{n}(S)=S$ and $\tau_{1}^{n}(S)=
S\cup \{n+1\}$. 
\end{example}

\subsection{Covariant Grothendieck constructions}
Let $\mathbb{B}$ be a small category and 
$\Psi:\mathbb{B}\rightarrow {\bf Cat}$ a functor.
Recall that the covariant {\bf Grothendieck construction} 
$\underset{\mathbb{B}}\int \Psi$ of $\Psi$ is the category 
with objects $(b,x)$ where $b\in \mathbb{B}$
and $x\in \Psi(b)$, and arrows  $(b,x)\rightarrow (b',y)$ are 
pairs $(u,f)$ with $u:b\rightarrow b'$ in $\mathbb{B}$ and 
$f:\Psi(u)(x)\rightarrow y$ in $\Psi(b')$. The projection 
$p:\underset{\mathbb{B}}\int \Psi\rightarrow \mathbb{B}$ 
is a split opfibration. The fiber category of $p$ over 
$b\in \mathbb{B}$ is isomorphic to $\Psi(b)$.
\\

As a rather trivial example, let $S$ be a nonempty set, 
viewed as a discrete category. The category $S_{+}$ (Example 1.1) 
is (also) a covariant Grothendieck construction. Indeed, let 
$2_{|S|}$ be the category with two objects $0$
and $1$ and $|S|$ arrows from $0$ to $1$. Let 
$\Psi:2_{|S|}\rightarrow {\bf Cat}$
be defined as $\Psi(0)=\ast$ and $\Psi(1)=S$. Then $S_{+}=
\underset{2_{|S|}}\int \Psi$. 

More interesting is the dual of Example 1.2.
\begin{example}
Let $F:I\rightarrow J$ be a functor between small 
categories. Let $\Psi(F):po\rightarrow {\bf Cat}$ be 
the functor 
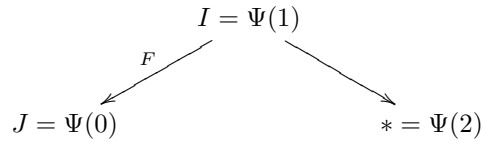
\[
   \xymatrix{
& I=\Psi(1) \ar[dr] \ar[dl]_{F}\\
J=\Psi(0) & & \ast=\Psi(2)\\
}
  \]
One may then form the covariant Grothendieck 
construction $\underset{po}\int \Psi(F)$. We shall 
use this construction when $F$ is the identity functor 
of a small category $J$. In this case we denote 
$\underset{po}\int \Psi(Id_{J})$ by $\underset{po}\int J$. 
When $J$ is the discrete category on a set $S$, 
$\underset{po}\int S$ maybe depicted as a spider 
with $|S|$ legs. We illustrate the case 
$S=\underline{4}$ (Example 1.4) 
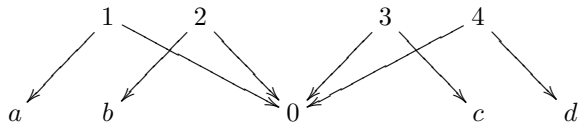
\[
   \xymatrix{
& 1 \ar[dl] \ar[drr] & 2 \ar[dr] \ar[dl] & & 3 \ar[dl] \ar[dr] & 
4 \ar[dr] \ar[dll]\\
a & b & & 0 & & c & d\\
}
  \]
We call the object $(2,\ast)$ of 
$\underset{po}\int S$ the body of the spider
($0$ in the above picture).
\end{example}
Let $J$ be a small category, $\mathcal{C}$ a cocomplete 
category and $\mathcal{Z}:(J_{+})^{op}\times J_{+}
\rightarrow \mathcal{C}$ a functor. For later purposes we 
need to understand $\overset{j\in J_{+}}\int \mathcal{Z}(j,j)$ 
better, especially when $J$ is a discrete category.
To begin with, we recall the calculation 
of $\overset{j\in J_{+}}\int \mathcal{Z}(j,j)$ as a colimit.
The (lower) {\bf twisted arrow category} of $J$, which we 
denote by $J_{\tau}$, has arrows $f:j\rightarrow k$ of $J$ as 
objects and a map $f\rightarrow f'$ is a commutative diagram 
\[
   \xymatrix{
j \ar[r] \ar[d]_{f} & j' \ar[d]^{f'}\\
k & k' \ar[l]\\
}
  \]
in $J$. There is a functor $K:(J_{+})_{\tau}\rightarrow 
(J_{+})^{op}\times J_{+}$, $K((j\rightarrow k))=(k,j)$, 
and $$\overset{j\in J_{+}}\int \mathcal{Z}(j,j)
\cong colim_{(J_{+})_{\tau}}\mathcal{Z}K$$
Now, if $J$ is a discrete category, there is a natural 
isomorphism $$\xi_{J}:\underset{po}\int J
\overset{\cong}\longrightarrow (J_{+})_{\tau}$$
given by $(0,j)\mapsto Id_{j}$, $(1,j)\mapsto 
(\emptyset \rightarrow j)$ and $(2,\ast)\mapsto 
Id_{\emptyset}$. Therefore we have a natural 
isomorphism $$colim_{\underset{po}\int J}\mathcal{Z}K\xi_{J}
\overset{\cong}\longrightarrow 
colim_{(J_{+})_{\tau}}\mathcal{Z}K$$

\begin{lem} {\rm (Inheritance results)}

$(a)$ If $J$ is a small Reedy category with cofibrant 
constants then so are $J_{+}$ and $\underset{pb}\int J$. 

$(b)$ If {\bf J} is a small subcategory of {\bf Cat} (which does 
not contain the empty category), then $\underset{pb}\int {\bf J}$ 
is a small subcategory of {\bf Cat} (which does not contain the
empty category).
\end{lem}
\begin{example}
Continuing Example 1.4, $\mathcal{P}_{0}(\underline{n+1})$ is a 
Reedy category with cofibrant constants if we let the inverse 
subcategory $\overleftarrow{\mathcal{P}_{0}(\underline{n+1})}$
be $\mathcal{P}_{0}(\underline{n+1})$, the direct subcategory 
$\overrightarrow{\mathcal{P}_{0}(\underline{n+1})}$
be the discrete category on the set of object of 
$\mathcal{P}_{0}(\underline{n+1})$ and the degree of 
$S\in \mathcal{P}_{0}(\underline{n+1})$ be $n+1-|S|$.
\end{example}
Lemma 1.6 has a dual formulation. In particular,
for every finite set $S$, $\underset{po}\int S$ is a 
Reedy category with fibrant constants.

\section{Homotopy limits and colimits I}
 
In order to construct the Goodwillie tower of
a functor between simplicial model categories, we 
recall in this section the minimum necessary 
from the theory of homotopy limits and colimits
in simplicial model categories
\subsection{Homotopy limits}
Let $\mathcal{C}$ be a category. We denote 
by $({\bf Cat}// \mathcal{C})$ the category 
with objects pairs $(I,\mathcal{X}:I\rightarrow \mathcal{C})$, 
where $I\in {\bf Cat}$, and arrows 
$<F,\alpha>:(I,\mathcal{X}:I\rightarrow \mathcal{C})
\rightarrow (J,\mathcal{Y}:J\rightarrow \mathcal{C})$
those pairs consisting of a functor $F:I\rightarrow J$ 
and a natural transformation 
$\alpha:\mathcal{Y}F\Rightarrow \mathcal{X}$. 
If $\mathcal{D}$ is another category, a functor 
$f:\mathcal{C}\rightarrow \mathcal{D}$
induces a functor $({\bf Cat}// f):({\bf Cat}// 
\mathcal{C})\rightarrow ({\bf Cat}// \mathcal{D})$. 
\\

Let $\mathcal{C}$ be a simplicial model category, 
$J$ a small category and $\mathcal{X}:J\rightarrow 
\mathcal{C}$ a functor. $holim_{J}\mathcal{X}$ stands 
for the homotopy limit of $\mathcal{X}$, 
as defined in \cite[18.1.8]{Hi}. $holim$ is a functor 
$({\bf Cat}// \mathcal{C})^{op} \rightarrow \mathcal{C}$.
$holim_{J}:\mathcal{C}^{J}\rightarrow \mathcal{C}$ 
is a simplicial functor (for $holim_{J}$ to be simplicial 
one does not need $\mathcal{C}^{J}$ be a  model category). 
A simplicial functor $F:\mathcal{C}\rightarrow \mathcal{D}$ 
between simplicial model categories induces a natural map 
$$F(holim_{J}\mathcal{X})\rightarrow holim_{J}F\mathcal{X}$$

Let $\hat{F}:\mathcal{C}\rightarrow \mathcal{C}$ be a 
fibrant approximation on $\mathcal{C}$. $cholim_{J}$
stands for the composite $$\mathcal{C}^{J}\overset{\hat{F}}
\rightarrow\mathcal{C}^{J}\overset{holim_{J}}\longrightarrow 
\mathcal{C}$$ and is referred to as the corrected homotopy limit 
of $J$-diagrams in $\mathcal{C}$. $cholim$ is a functor 
$({\bf Cat}// \mathcal{C})^{op}\rightarrow \mathcal{C}$. 
We shall use both $holim$ and $cholim$. If $\mathcal{C}$ is 
cofibrantly generated, then $\mathcal{C}$ has a simplicial 
fibrant approximation \cite[4.3.7]{Hi}, and, with this choice 
of fibrant approximation, $cholim_{J}$ is a simplicial functor.

\subsection{Homotopy colimits}
Let $\mathcal{C}$ be a category. We denote by
$({\bf Cat}\downarrow \mathcal{C})$ the category with 
objects pairs $(I,\mathcal{X}:I\rightarrow \mathcal{C})$, 
where $I\in {\bf Cat}$, and arrows 
$<F,\alpha>:(I,\mathcal{X}:I\rightarrow \mathcal{C})
\rightarrow (J,\mathcal{Y}:J\rightarrow \mathcal{C})$
those pairs consisting of a functor $F:I\rightarrow J$ 
and a natural transformation $\alpha:\mathcal{X}
\Rightarrow \mathcal{Y}F$. If $\mathcal{D}$ is another 
category, a functor $f:\mathcal{C}\rightarrow 
\mathcal{D}$ induces a functor 
$({\bf Cat}\downarrow f):({\bf Cat}\downarrow \mathcal{C})
\rightarrow ({\bf Cat}\downarrow \mathcal{D})$.
\\

Let $\mathcal{C}$ be a simplicial model category, 
$J$ a small category and $\mathcal{X}:J\rightarrow 
\mathcal{C}$ a functor. $hocolim_{J}\mathcal{X}$ stands 
for the homotopy colimit of $\mathcal{X}$, as defined in 
\cite[18.1.2]{Hi}. $hocolim$ is a functor 
$({\bf Cat}\downarrow \mathcal{C}) \rightarrow \mathcal{C}$.
A simplicial functor $F:\mathcal{C}\rightarrow \mathcal{D}$ 
induces a natural map $$hocolim_{J}F\mathcal{X}\rightarrow 
F(hocolim_{J}\mathcal{X})$$ Let $\tilde{C}:\mathcal{C}
\rightarrow \mathcal{C}$ be a cofibrant approximation 
on $\mathcal{C}$. $chocolim_{J}$ stands for the composite 
$$\mathcal{C}^{J}\overset{\tilde{C}}\rightarrow
\mathcal{C}^{J}\overset{hocolim_{J}}\longrightarrow 
\mathcal{C}$$ and is referred to as the corrected 
homotopy colimit of $J$-diagrams in $\mathcal{C}$.
$chocolim$ is a functor $({\bf Cat}\downarrow 
\mathcal{C}) \rightarrow \mathcal{C}$. 
We shall use both $hocolim$ and $chocolim$. 

\section{The construction of the Taylor tower}

In this section we examine the construction of the Taylor 
tower of a functor between simplicial model categories 
\cite{Go3},\cite{Ku}. After that we study the 
ingredients which appear in the construction.

\subsection{The $\star$ operation}
We begin with some general considerations. 
\\

(1) Let $\mathcal{V}$ be a monoidal category and $\mathcal{C}$
a category equipped with an action $\circledast:\mathcal{V}\times 
\mathcal{C} \rightarrow \mathcal{C}$ of $\mathcal{V}$. Let $A$ 
be a monoid in $\mathcal{V}$ and $Y$ a left $A$-module in $\mathcal{C}$.
The category $(\mathcal{V}\downarrow A)$ becomes a monoidal
category which acts on $(\mathcal{C}\downarrow Y)$. Let us denote 
this action by $\circledast '$. Suppose that $\mathcal{C}$ is cocomplete.
For each small category $J$ we have an induced functor
$$\circledast'_{J}:(\mathcal{V}\downarrow A)^{J^{op}}\times 
(\mathcal{C}\downarrow Y)^{J}\longrightarrow 
(\mathcal{C}\downarrow Y)$$

(2) Let $\mathcal{C}$ be a category with terminal object and 
$J$ a small category. We define a functor
$R_{J}:\mathcal{C}\rightarrow \mathcal{C}^{J_{+}}$ as
$$R_{J}X(j)=
\begin{cases}
X, & \text{if }  j=\emptyset\\
\ast, & \text{otherwise}
\end{cases} $$
In particular, for any category $\mathcal{C}$ and any $Y\in 
\mathcal{C}$ we have the functor $R_{J}:(\mathcal{C}\downarrow Y)
\rightarrow (\mathcal{C}\downarrow Y)^{J_{+}}$.
\\

We shall apply the previous considerations in the following 
situation. $\mathcal{C}$ is a simplicial model category, $A$
is the terminal simplicial monoid and $Y\in \mathcal{C}$.
For each small category $J$ we have then a composite functor
$${\bf S}^{J_{+}^{op}}\times (\mathcal{C}\downarrow Y)\overset{Id\times R_{J}}
\longrightarrow {\bf S}^{J_{+}^{op}}\times (\mathcal{C}\downarrow Y)^{J_{+}}
\overset{\circledast '_{J_{+}}}\longrightarrow (\mathcal{C}\downarrow Y)$$
Let $(X,X\rightarrow Y)\in (\mathcal{C}\downarrow Y)$. We define 
$$J\star_{Y} X=hocolim_{J_{+}} R_{J}(X,X\rightarrow Y)$$
We obtain a functor $$-\star_{Y}-:{\bf Cat}\times 
(\mathcal{C}\downarrow Y)\rightarrow (\mathcal{C}\downarrow Y)$$
When $Y=\ast$ we denote $J\star_{\ast} X$ by $J\star X$, so that
$$J\star X=hocolim_{J_{+}}R_{J}X$$
One has $\emptyset \star X=X$ and $\ast \star X=CX$, the cone on $X$.
The $\star$ operation can be restricted to any subcategory 
of {\bf Cat} which does not contain the empty category. In
particular, for each $n\geq 0$ we have a functor
$-\star-:\mathcal{P}_{0}(\underline{n+1})\times \mathcal{C}
\rightarrow \mathcal{C}$.

\subsection{The Taylor tower}
Let {\bf J} be a small subcategory of {\bf Cat} which does 
not contain the empty category.
Let $\mathcal{C}$ and $\mathcal{D}$ be two simplicial
model categories. For each $n\geq 0$ we define the 
endofunctor $$T_{n,{\bf J}}:Fun(\mathcal{C},\mathcal{D})
\rightarrow Fun(\mathcal{C},\mathcal{D})$$
as $$(T_{n,{\bf J}}F)(X)=cholim_{{\bf J}(n+1)}F((-)\star X)$$
We obtain a natural
transformation $t_{n}:Id\Rightarrow T_{n,{\bf J}}$. 
Using the map $\tau_{0}^{n+1}:{\bf J}(n+1)\rightarrow 
{\bf J}(n+2)$ (see below Construction 1.3) we obtain a natural 
transformation $q_{n+1,1}:T_{n+1,{\bf J}}\Rightarrow T_{n,{\bf J}}$ 
such that $q_{n+1,1}t_{n+1,{\bf J}}=t_{n,{\bf J}}$ for each $n\geq 0$.

\begin{observation}
Let $(\mathcal{V},\otimes ,e)$ be a monoidal category. Let $T$
be an object of $\mathcal{V}$ and $t:e\rightarrow T$ a map. 
One can construct the functor $\omega \rightarrow \mathcal{V}$
$$e\rightarrow T\cong e\otimes T\overset{t\otimes T}
\longrightarrow T\otimes T\cong e\otimes T\otimes T
\overset{t\otimes T^{\otimes 2}}\longrightarrow
T^{\otimes 3}\cong e\otimes T^{\otimes 3}\rightarrow ...$$
In particular, let $cst(e)$ denote the unit object of 
$\mathcal{V}^{\omega^{op}}$ and let $T$ be an object of 
$\mathcal{V}^{\omega^{op}}$, with structure maps 
$q_{n+1,1}:T_{n+1}\rightarrow T_{n}$. 
The observation we want to make 
is that a map $t:cst(e)\rightarrow T$ gives rise to a functor 
$\omega \times \omega^{op}\rightarrow \mathcal{V}$,
$(i,n)\mapsto T_{n}^{\otimes i}$, with the convention 
that for every $X\in \mathcal{V}$, $X^{\otimes 0}=e$.
\end{observation}
We shall apply the observation to the situation
$\mathcal{V}=END(Fun(\mathcal{C},\mathcal{D}))$, $T_{-,{\bf J}}$
defined above (an object of $\mathcal{V}^{\omega^{op}}$) and 
$t:cst(Id)\rightarrow T_{-,{\bf J}}$ defined above 
(a map of $\mathcal{V}^{\omega^{op}}$).
Then, setting $$P_{n,{\bf J}}=chocolim_{\omega}(Id\rightarrow T_{n,{\bf J}} 
\rightarrow T_{n,{\bf J}}^{2} \rightarrow  T_{n,{\bf J}}^{3}\rightarrow ...)$$
defines an object of $\mathcal{V}$ and there are natural maps
$q_{n+1}:P_{n+1,{\bf J}}\Rightarrow P_{n,{\bf J}}$, so that 
we obtain an object $P_{-,{\bf J}}$ of $\mathcal{V}^{\omega^{op}}$. 
If $p_{n}:Id\Rightarrow P_{n,{\bf J}}$
is the natural map, then clearly $q_{n+1}p_{n+1}=p_{n}$, 
that is, $p:cst(Id)\Rightarrow P_{-,{\bf J}}$.

\begin{definition}
Let $\mathcal{C}$ and $\mathcal{D}$ be two simplicial 
model categories and $F\in Fun(\mathcal{C},\mathcal{D})$.
Let {\bf J} be a small subcategory of {\bf Cat} which does 
not contain the empty category. $P_{n,{\bf J}}F$ is referred 
to as the $n$-th {\bf Taylor polynomial} of $F$ 
{\bf with respect to J}, and $$...\rightarrow 
P_{n,{\bf J}}F\overset{q_{n}F}\rightarrow
 P_{n-1,{\bf J}}F\overset{q_{n-1}F}
\rightarrow...\rightarrow  P_{1,{\bf J}}F
\overset{q_{1}F}\rightarrow P_{0,{\bf J}}F$$
as the {\bf Taylor tower} of $F$ {\bf with respect to J}. 
We write it as $\{P_{n,{\bf J}}F\}$. When {\bf J} consists 
of the terminal category only, $P_{n,\ast}F$ is referred 
to as the $n$-th {\bf Taylor polynomial} of $F$ and 
written $P_{n}F$. The Taylor tower of $F$ with respect 
to $\ast$ is written $\{P_{n}F\}$.
\end{definition}
\begin{remark}
To construct the Taylor tower of a functor 
$F:\mathcal{C}\rightarrow \mathcal{D}$ one does 
not need the full strength of the fact that $\mathcal{C}$
is a simplicial model category. One only needs a category 
$\mathcal{C}$ with terminal object such that for every 
finite set $S$ and every $X\in \mathcal{C}$, $S
\star X \in \mathcal{C}$. In light of the 
considerations from 1.2, some full subcategories 
of a locally presentable pointed simplicial model 
category consisting of cofibrant objects
are natural candidates for such a $\mathcal{C}$.
For example, let $\mathcal{C}$ be a locally 
$\lambda$-presentable pointed simplicial model 
category such that for every finite set $S$, tensoring
with $N( S_{+})^{op}$ preserves $\lambda$-presentable
objects. We denote by $\mathcal{C}_{\lambda, c}$ 
the full subcategory of $\mathcal{C}$ consisting of the
$\lambda$-presentable objects which are cofibrant.
Given a simplicial model category $\mathcal{D}$ and a functor
$F:\mathcal{C}_{\lambda, c}\rightarrow \mathcal{D}$,
one can construct $P_{n}F:\mathcal{C}_{\lambda, c}
\rightarrow \mathcal{D}$.
\end{remark}

\subsection{Properties of the $\star$ operation}

(1) For each $J\in {\bf Cat}$, $J\star -$ is a simplicial
functor. $J\star -$ preserves the simplicial action provided
that $\mathcal{C}$ is pointed.

(2) Let $I$ and $J$ be small categories and $\mathcal{X}:
I\rightarrow \mathcal{C}$ a functor. The natural map 
$$hocolim_{I}(J\star \mathcal{X}(-))\rightarrow 
J\star (hocolim_{I}\mathcal{X})$$
is an isomorphism provided that $\mathcal{C}$ is pointed.
Consequently, in this case one has $I\star (J\star X)\cong J\star (I\star X)$, 
so that $(T_{n}F)(J\star (-))\cong T_{n}(F(J\star (-)))$ and therefore 
$(P_{n}F)(J\star (-))\cong P_{n}(F(J\star (-)))$ for every 
$F\in Fun(\mathcal{C},\mathcal{D})$.

(3) Let $F:\mathcal{C}\rightarrow \mathcal{D}$ and 
$G:\mathcal{D}\rightarrow \mathcal{E}$ be functors between
simplicial model categories. Then
$$(J\star G(-))F= J\star GF(-):\mathcal{C}\rightarrow \mathcal{E}$$

(4) Let $F:\mathcal{C}\rightarrow \mathcal{D}$ be a functor
between simplicial model categories. We have a natural transformation 
$F^{J_{+}}R_{J}\Rightarrow R_{J}F$
which induces a natural transformation
$$hocolim_{J_{+}}F^{J_{+}}R_{J}(-) 
\Rightarrow J\star F(-):\mathcal{C}\rightarrow \mathcal{D}$$
which is a natural isomorphism if $F(\ast)\cong \ast$. If, moreover,
$F$ is a simplicial functor, then we have a natural transformation
$$hocolim_{J_{+}}F^{J_{+}}R_{J}(-) 
\Rightarrow F(J\star (-)):\mathcal{C}\rightarrow \mathcal{D}$$
Summing up, if $F$ is a simplicial functor and $F(\ast)\cong \ast$,
then we have a natural map
$$ J\star F(-)\Rightarrow  F(J\star (-))$$

(5) Let $F:\mathcal{C}\rightarrow \mathcal{D}$ and 
$G:\mathcal{D}\rightarrow \mathcal{E}$ be simplicial functors between
simplicial model categories. We have natural transformations
$(GF)^{J_{+}}R_{J}\Rightarrow G^{J_{+}}R_{J}F\Rightarrow R_{J}(GF)$
which, together with (4), induce a diagram of natural transformations
\[
   \xymatrix{
& hocolim_{J_{+}}(GF)^{J_{+}}R_{J}(-) \ar[ddl] \ar[dr]\\
& & G(hocolim_{J_{+}}F^{J_{+}}R_{J}(-)) \ar[dl] \ar[dr]\\
J\star GF(-) & G(J\star F(-)) & & GF(J\star (-))\\
}
  \]
If, moreover, $F(\ast)\cong \ast$ and $G(\ast)\cong \ast$, 
then by (4) we have natural maps $$J\star GF(-) \Rightarrow 
G(J\star F(-))\Rightarrow GF(J\star (-))$$

\subsection{Elementary properties of the $T_{n}$ and $P_{n}$ 
constructions} Let $F\in Fun(\mathcal{C},\mathcal{D})$. 

(1) Suppose that $\mathcal{D}$ is endowed with simplicial
fibrant and cofibrant approximation functors, and that we 
agree to construct the $cholim$ and $chocolim$ functors using 
these approximations. If $F$ is a simplicial functor, then, using 
the first part of 3.3(1), $T_{n}F$ and $P_{n}F$ are simplicial
functors as well.

(2) Suppose that the terminal object of $\mathcal{C}$ is cofibrant.
If $F$ preserves weak equivalences between cofibrant objects
then so do $T_{n}F$ and $P_{n}F$.

(3) Let $G\in Fun(\mathcal{C},\mathcal{D})$ and let 
$\alpha:F\Rightarrow G$ be a map which is objectwise a weak 
equivalence. Then the induced maps $T_{n}\alpha:T_{n}F
\Rightarrow T_{n}G$ and $P_{n}\alpha:P_{n}F\Rightarrow 
P_{n}G$ are objectwise a weak equivalence. Suppose that 
the terminal object of $\mathcal{C}$ is cofibrant. If 
$\alpha$ is objectwise a weak equivalence on cofibrant 
objects, then so are $T_{n}\alpha$ and $P_{n}\alpha$.

(4) Let use denote the tensor in $\mathcal{C}$ by $-\circledast -$ 
and the cotensor in $\mathcal{D}$ by $(-)^{(-)}$. One has 
$$T_{n}F(\ast)\cong cholim_{S\in \mathcal{P}_{0}(\underline{n+1})} 
F(N(S_{+}^{op})\circledast \ast)$$
If $\mathcal{C}$ is pointed then $T_{n}F(0)\cong 
(\hat{F}F0)^{N(\mathcal{P}_{0}(\underline{n+1}))}$.

(5) Suppose that $F$ preserves weak equivalences. 
Then for each $X\in \mathcal{C}$,
$P_{0}F(X)$ has the homotopy type of $F(\ast)$. Consequently,
the map $p_{0}F:F\Rightarrow P_{0}F$ is an objectwise 
weak equivalence if and only if $F(X)\rightarrow F(\ast)$ is a weak
equivalence for each $X\in \mathcal{C}$ (if and only if $F$ sends
every map in $\mathcal{C}$ to a weak equivalence in $\mathcal{D}$).

A less elementary property of the $T_{n}$ and $P_{n}$ 
constructions is given in Corollary 5.6.
\section{The auxiliary tower}
Let $\mathcal{C}$ and $\mathcal{D}$ be two simplicial 
model categories and $F\in Fun(\mathcal{C},\mathcal{D})$.
In this section we construct a tower $\{\tilde{P}_{n}F\}$ 
and a map of towers $\{\tilde{P}_{n}F\}\rightarrow \{P_{n}F\}$.
If the terminal object of $\mathcal{C}$ is cofibrant and 
$F$ preserves weak equivalences between cofibrant objects,
this map is shown to be a weak equivalence when 
evaluated at cofibrant objects.

\subsection{The $\star^{h}$ operation}
Let $\mathcal{C}$ be a simplicial model category
with simplicial action which we denote by $\circledast$.
Let $S$ be a set, viewed as a discrete category and $X\in \mathcal{C}$.
Recall from 3.1 the functor $R_{S}$. We define a functor 
$(S,X):(S_{+})^{op}\times S_{+}\rightarrow \mathcal{C}$ as 
$(S,X)(j,k)=N(j\downarrow S_{+})^{op}\circledast R_{S}X(k)$.
From 1.2 and 3.1 we have 
$$colim_{\underset{po}\int S}(S,X)K\xi_{S}\overset{\cong}
\longrightarrow \overset{j\in S_{+}}\int (S,X)\cong S\ast X$$
For simplicity, we define $\tilde{R}_{S}:\mathcal{C}
\rightarrow \mathcal{C}^{\underset{po}\int S}$ 
as $\tilde{R}_{S}X=(S,X)K\xi_{S}$. Precisely,
$\tilde{R}_{S}X(0,k)=N(k\downarrow S_{+})^{op}\circledast \ast=\ast$,
$\tilde{R}_{S}X(1,j)=N(j\downarrow S_{+})^{op}\circledast X=X$ 
and $\tilde{R}_{S}X(2,\ast)=N( S_{+})^{op}\circledast X$.
By definition
$$S\star^{h} X=hocolim_{\underset{po}\int S}\tilde{R}_{S}X$$

\subsection{The auxiliary tower}
Let $\mathcal{C}$ and $\mathcal{D}$ be two 
simplicial model categories. For each $n\geq 0$ 
we define the endofunctor 
$$\tilde{T}_{n}:Fun(\mathcal{C},\mathcal{D})
\rightarrow Fun(\mathcal{C},\mathcal{D})$$
as $$(\tilde{T}_{n}F)(X)=cholim_{\mathcal{P}_{0}
(\underline{n+1})}F((-)\star^{h} X)$$ We obtain a 
natural transformation $\tilde{t}_{n}:Id\Rightarrow \tilde{T}_{n}$. 
Using the map $\tau_{0}^{n+1}:\mathcal{P}_{0}(\underline{n+1})
\rightarrow \mathcal{P}_{0}(\underline{n+2})$
(Example 1.4) we obtain a natural transformation 
$\tilde{q}_{n+1,1}:\tilde{T}_{n+1}\Rightarrow 
\tilde{T}_{n}$ such that $\tilde{q}_{n+1,1}\tilde{t}_{n+1}=\tilde{t}_{n}$ 
for each $n\geq 0$. We apply Observation 3.1 to the situation
$\mathcal{V}=END(Fun(\mathcal{C},\mathcal{D}))$, $\tilde{T}$
defined above (an object of $\mathcal{V}^{\omega^{op}}$) and 
$\tilde{t}:cst(Id)\rightarrow \tilde{T}$ defined above 
(a map of $\mathcal{V}^{\omega^{op}}$).
Then, setting $$\tilde{P}_{n}=chocolim_{\omega}
(Id\rightarrow \tilde{T}_{n} \rightarrow \tilde{T}_{n}^{2} 
\rightarrow \tilde{T}_{n}^{3}\rightarrow ...)$$
defines an object of $\mathcal{V}$ and there are natural 
maps $\tilde{q}_{n+1}:\tilde{P}_{n+1}\Rightarrow 
\tilde{P}_{n}$, so that we obtain an object 
$\tilde{P}$ of $\mathcal{V}^{\omega^{op}}$. If 
$\tilde{p}_{n}:Id\Rightarrow \tilde{P}_{n}$ is the natural map, 
then clearly $\tilde{q}_{n+1}\tilde{p}_{n+1}=\tilde{p}_{n}$, 
that is, $\tilde{p}:cst(Id)\Rightarrow \tilde{P}$.
\begin{lem}
Suppose that the terminal object of $\mathcal{C}$ is 
cofibrant and $F\in Fun(\mathcal{C},\mathcal{D})$ 
preserves weak equivalences between cofibrant objects.
Then the natural map of towers $\{\tilde{P}_{n}F\}
\rightarrow \{P_{n}F\}$ is a weak equivalence when 
evaluated at cofibrant objects.
\end{lem}
\begin{proof}
Let $X\in \mathcal{C}$ be cofibrant. By hypothesis, it 
suffices to show that for every finite set $S$, the map 
$S\star^{h} X\rightarrow S\star X$ is a weak equivalence. 
For this, it suffices by \cite[19.9.1(1)]{Hi} to show that 
$\tilde{R}_{S}X$ is Reedy cofibrant. $\tilde{R}_{S}X$
is a spider in $\mathcal{C}$ with $|S|$ legs (Example 1.5)
whose body is $N( S_{+})^{op}\circledast X$. It
is cofibrant since the natural map $\underset{S}\coprod 
X\rightarrow N( S_{+})^{op}\circledast X$ is a cofibration.
\end{proof}
It seems natural to us to relate Rezk's proof \cite{Re}
of Goodwillie's $n$-excisive approximation theorem 
\cite[Theorem 1.8]{Go3} to the tower $\{\tilde{P}_{n}F\}$.
Let $\mathcal{X}:\mathcal{P}(\underline{n+1})
\rightarrow \mathcal{C}$. For $U,T\in 
\mathcal{P}(\underline{n+1})$ we have 
$$U\star^{h} \mathcal{X}(T)=
hocolim_{\underset{po}\int U}\tilde{R}_{U}\mathcal{X}(T)$$
Following Rezk, we define $\mathcal{X}^{2}:
\mathcal{P}(\underline{n+1})\times 
\mathcal{P}(\underline{n+1})\rightarrow \mathcal{C}$
as $$\mathcal{X}^{2}(U,T)=hocolim_{\underset{po}\int U}
\underline{\mathcal{X}}_{T}$$
where $\underline{\mathcal{X}}_{T}(0,s)=
\mathcal{X}(T\cup \{s\})$, $\underline{\mathcal{X}}_{T}(1,s)
=\mathcal{X}(T)$ and $\underline{\mathcal{X}}_{T}(2,\ast)=
\mathcal{X}(T)$. We have natural maps
$$\mathcal{X}^{2}(U,T)\longrightarrow U\star^{h} \mathcal{X}(T)$$
and $$\mathcal{X}_{U}(T)\overset{def}=hocolim(
\mathcal{X}(T)\leftarrow \underset{U}\coprod \mathcal{X}(T)
\rightarrow \underset{U}\coprod\mathcal{X}(T\cup \{s\}))
\longrightarrow \mathcal{X}^{2}(U,T)$$

\section{Homotopy limits and colimits II}

This section is a continuation of section 2.
We review here more elaborated results 
on (corrected) homotopy limits and colimits
in simplicial model categories. 

\subsection{Homotopy limits}
Recall that, in a simplicial model category, a diagram 
\[
\xymatrix{
W \ar[r] \ar[d] & X \ar[d]\\
Z \ar[r] & Y\\
}
  \] 
with $X,Y$ and $Z$ fibrant objects is homotopy Cartesian 
in the sense of  \cite[Definition 1.3]{Go2} if and only if it is a 
homotopy pullback in the sense of \cite[Chapter 7]{Ho1}.

Let $\mathcal{C}$ be a pointed simplicial model category. If 
$g:X\rightarrow Y$ is a map in $\mathcal{C}$, we denote by 
{\rm chf}$(g)$ the corrected homotopy limit of the diagram 
$X\overset{g}\rightarrow Y\leftarrow 0$. This defines a
functor ${\rm chf}:\mathcal{C}^{[1]}\rightarrow \mathcal{C}$. 
For each small category $J$, there is then a functor ${\rm chf}:
(\mathcal{C}^{J})^{[1]}\rightarrow \mathcal{C}^{J}$
given as ${\rm hf}(g)_{j}={\rm chf}(g_{j})$. We denote by 
{\rm hf}$(g)$ the homotopy limit of the diagram 
$X\overset{g}\rightarrow Y\leftarrow 0$. This defines a
functor ${\rm hf}:\mathcal{C}^{[1]}\rightarrow \mathcal{C}$. 

Let $F:\mathcal{C}\rightarrow \mathcal{D}$ be a functor
between pointed simplicial model categories.
We have maps $$F{\rm chf}(g)\rightarrow FX\overset{ch}\times_{FY} 
F0 \leftarrow {\rm chf}(Fg)$$ where $FX\overset{ch}\times_{FY} F0$
is the corrected homotopy limit of the diagram 
$FX\overset{Fg}\rightarrow FY\leftarrow F0$.
The two maps displayed above are weak equivalences
if, for example, $F0\rightarrow 0$ is a weak equivalence 
and $F$ sends homotopy pullback diagrams of the form 
\[
   \xymatrix{
A \ar[r] \ar[d] & 0 \ar[d]\\
B \ar[r] & C\\
}
  \]
to homotopy pullbacks.

The next result is closely related to \cite[Lemma 1.3.2(c)]{Sc}.
\begin{lem}
Let $I$ be a small filtered category and 
$\mathcal{C}$ a locally finitely presentable simplicial 
model category whose tensor, viewed as a functor
${\bf S}\times \mathcal{C}\rightarrow \mathcal{C}$, 
preserves finitely presentable objects, and such that an 
$I$-indexed colimit of weak equivalences of $\mathcal{C}$ 
is a weak equivalence and an $I$-indexed 
colimit of fibrant objects of $\mathcal{C}$ is fibrant. 
Let $J$ be a finite category such that for each object 
$j$ of $J$, the nerve of $(J\downarrow j)$ is finitely 
presentable. Then for every $\mathcal{X}:I\rightarrow 
\mathcal{C}^{J}$, the natural map $$colim_{I}
cholim_{J}\mathcal{X}\rightarrow cholim_{J}
(colim_{I}\mathcal{X})$$ is a weak equivalence. 
\end{lem}
\begin{proof}
The natural map is the composite
$$colim_{I}cholim_{J}\mathcal{X}=
colim_{I}holim_{J}\hat{F}\mathcal{X}\overset{\cong}
\longrightarrow holim_{J}(colim_{I}\hat{F}\mathcal{X})
\rightarrow holim_{J}\hat{F}(colim_{I}\mathcal{X})
=cholim_{J}(colim_{I}\mathcal{X})$$
The first map is an isomorphism by Lemma 5.2.
The second map can be seen to be a weak equivalence 
using the commutative diagram
\[
   \xymatrix{
colim_{I}\hat{F}\mathcal{X} \ar[rr] & &
\hat{F}(colim_{I}\mathcal{X})\\
&  colim_{I}\mathcal{X} \ar[ul] \ar[ur]\\
}
  \]
and the other assumptions on $\mathcal{C}$.
\end{proof}
\begin{lem}
Let $\mathcal{C}$ be a locally finitely presentable 
simplicial model category whose tensor, viewed as a 
functor ${\bf S}\times \mathcal{C}\rightarrow \mathcal{C}$, 
preserves finitely presentable objects. Let $J$ be a finite 
category such that for each object $j$ of $J$, the nerve of 
$(J\downarrow j)$ is finitely presentable. Then 
$holim_{J}$ preserves filtered colimits.
\end{lem}
\begin{proof}
The fact that $\mathcal{C}$ is a \emph{model} category
is not relevant. Perhaps the simplest proof
is to write $holim_{J}$ as a limit indexed over the (upper)
twisted arrow category of $J$ and to use adjunctions 
and standard properties of locally presentable categories.
\end{proof}
We recall \cite[Lemma 4.3]{Ho2} that, in an almost finitely 
generated model category, an $\omega$-indexed colimit of 
weak equivalences is a weak equivalence and an 
$\omega$-indexed colimit of fibrant objects is fibrant. 
\\

Let now $p:\mathbb{E}\rightarrow \mathbb{B}$ a split fibration
between small categories. We denote by $\mathbb{E}_{b}$ the
fiber category over $b\in \mathbb{B}$ and by
$\iota_{b}:\mathbb{E}_{b}\rightarrow \mathbb{E}$ the natural 
functor. Let $\mathcal{C}$ be a simplicial model category and 
$\mathcal{X}:\mathbb{E}\rightarrow \mathcal{C}$ a functor.
We obtain a functor $$\mathbb{B}\rightarrow
({\bf Cat}// \mathcal{C}), \qquad
b\mapsto \mathcal{X}_{b}:=\mathcal{X}\iota_{b}$$
Therefore there is a natural map
$$cholim_{\mathbb{E}}\mathcal{X}\rightarrow 
holim_{\mathbb{B}}cholim_{\mathbb{E}_{b}}\mathcal{X}_{b}$$
For $b\in \mathbb{B}$ we denote by $q$ the natural functor
$(b\downarrow p)\rightarrow \mathbb{E}$.
\begin{theorem} \cite[A dual of theorem 26.8]{CS}
Let $p:\mathbb{E}\rightarrow \mathbb{B}$ a split fibration between
small Reedy categories with cofibrant constants such that
$(i)$ $p$ is a morphism of Reedy categories, $(ii)$ $(p^{\ast},
p_{\ast})$ is a Quillen pair and $(iii)$ for each $b\in \mathbb{B}$,
$(q_{!},q^{\ast})$ is a Quillen pair.
Let $\mathcal{C}$ be a simplicial model category
and $\mathcal{X}:\mathbb{E}\rightarrow \mathcal{C}$
a functor. Then the natural map $$cholim_{\mathbb{E}}\mathcal{X}
\rightarrow holim_{\mathbb{B}}cholim_{\mathbb{E}_{b}}
\mathcal{X}_{b}$$ is a weak equivalence.
\end{theorem}
Theorem 5.3 will be applied to split fibrations with base $pb$.

\subsection{Homotopy colimits}
Let $I$ be a small filtered category and $\phi:I\rightarrow {\bf Cat}$ 
a functor. Let $J=colim_{I}\phi$ and let $u_{i}:\phi_{i}\rightarrow J$ be 
the canonical map $(i\in I)$. If $\mathcal{X}:J\rightarrow \mathcal{C}$ 
is a functor, then \cite[XII 3.5]{BK} $$hocolim_{J}\mathcal{X}\cong 
colim_{I}(hocolim_{\phi_{i}}\mathcal{X}/i)$$
where $\mathcal{X}/i$ is the composite $\phi_{i}\overset{u_{i}}
\rightarrow J\overset{\mathcal{X}}\rightarrow \mathcal{C}$. 
In particular, let $J$ be a small Reedy category and $F^{i}J$ 
the $i$-filtration of $J$ \cite[15.1.22]{Hi}. Consider 
$\phi:\omega\rightarrow {\bf Cat}$, $\phi_{i}=F^{i}J$. 
Then $J=colim_{\omega}\phi$ \cite[15.1.25]{Hi},
hence $hocolim_{J}\mathcal{X}\cong 
colim_{\omega}(hocolim_{F^{i}J}\mathcal{X}/i)$.
\begin{lem}
Let $F:\mathcal{C}\rightarrow \mathcal{D}$ be a
functor between simplicial model categories
which preserves weak equivalences. Let $I$
be a small filtered category such that an $I$-indexed 
colimit of weak equivalences of $\mathcal{D}$
is a weak equivalence. Suppose that $F$
preserves $I$-indexed colimits. Then for every
$\mathcal{X}:I\rightarrow \mathcal{C}$,
$chocolim_{I}F\mathcal{X}$ and 
$F(chocolim_{I}\mathcal{X})$ are weakly
equivalent. If $\mathcal{C}=\mathcal{D}$ and 
there is a natural transformation $Id\Rightarrow F$
which is objectwise a weak equivalence, the 
requirement that $F$ preserves $I$-indexed 
colimits can be dropped.
\end{lem}
\begin{proof}
For the first part we have chains of arrows
$$hocolim_{I}\tilde{C}F\mathcal{X}\cong 
colim_{I}(hocolim_{(I\downarrow i)}\tilde{C}F\mathcal{X}/i)
\rightarrow colim_{I}\tilde{C}F\mathcal{X}(i)
\rightarrow colim_{I}F\mathcal{X}(i)\leftarrow 
colim_{I}F\tilde{C}\mathcal{X}(i)$$
and $$F(hocolim_{I}\tilde{C}\mathcal{X})\cong 
F(colim_{I}(hocolim_{(I\downarrow i)}\tilde{C}\mathcal{X}/i))
\overset{\cong}\leftarrow colim_{I}F(hocolim_{(I\downarrow i)}
\tilde{C}\mathcal{X}/i))\rightarrow colim_{I}F\tilde{C}\mathcal{X}(i)$$
Each arrow in the above chains is a weak equivalence 
by \cite[19.6.8(1)]{Hi} and hypothesis. For the second part,
the isomorphism in the second chain of arrows is a weak 
equivalence.
\end{proof}
Lemma 5.4 gives sufficient conditions for a functor 
to be `$I$-finitary' \cite[Definition 5.10]{Go3}. 
The next result, which in common parlance 
says that certain homotopy limits `commute' with 
filtered homotopy colimits, is a consequence of 
Lemmas 5.4 and 5.1.
\begin{corollary}
Let $I$ be a small filtered category and $\mathcal{C}$ 
a locally finitely presentable simplicial 
model category whose tensor, viewed as a functor
${\bf S}\times \mathcal{C}\rightarrow \mathcal{C}$, 
preserves finitely presentable objects, and that an 
$I$-indexed colimit of weak equivalences of $\mathcal{C}$ 
is a weak equivalence and an $I$-indexed 
colimit of fibrant objects of $\mathcal{C}$ is fibrant. 
Let $J$ be a finite category such that for each object 
$j$ of $J$, the nerve of $(J\downarrow j)$ is finitely 
presentable. Then for every $\mathcal{X}:I\rightarrow 
\mathcal{C}^{J}$,
$chocolim_{I}cholim_{J}\mathcal{X}$ and 
$cholim_{J}chocolim_{I}\mathcal{X}$ are weakly
equivalent.
\end{corollary}
Let $\mathcal{C}$ and $\mathcal{D}$ be two 
simplicial model categories, $I$ a small category
and $F:I\rightarrow Fun(\mathcal{C},\mathcal{D})$.
We define $chocolim_{I}F:\mathcal{C}\rightarrow
\mathcal{D}$ as $(chocolim_{I}F)(X)=chocolim_{I}F(X)$.
\begin{corollary}
Let $\mathcal{C}$ and $\mathcal{D}$ be two 
simplicial model categories and $I$ a small filtered category. 
Suppose that $\mathcal{D}$ is locally presentable, that its 
tensor, viewed as a functor ${\bf S}\times \mathcal{D}
\rightarrow \mathcal{D}$, preserves finitely presentable 
objects, and that an $I$-indexed colimit of weak 
equivalences of $\mathcal{D}$ is a weak equivalence 
and an $I$-indexed colimit of fibrant objects of $\mathcal{D}$ 
is fibrant. Then, for every $F:I\rightarrow Fun(\mathcal{C},\mathcal{D})$,
$T_{n}(chocolim_{I}F)$ and $chocolim_{I}T_{n}F$ are objectwise 
weakly equivalent. Consequently, $P_{n}(chocolim_{I}F)$ and 
$chocolim_{I}P_{n}F$ are objectwise weakly equivalent. 
\end{corollary}

\section{Generalized homotopy Cartesian cubes}
In this section we give an analogue, in our context, of 
some useful results from \cite{Go2},\cite{Ku},
\cite{Ch} and \cite{Go3} which revolve around the notion 
of homotopy Cartesian cube \cite[Definition 1.3]{Go2}. 
\begin{definition}
Let $\mathcal{C}$ be a simplicial model category 
and $J$ a small category. An object $\mathcal{X}$ of 
$\mathcal{C}^{J_{+}}$ is {\bf homotopy Cartesian} if the 
natural map $$\mathcal{X}(\emptyset) \rightarrow 
cholim_{J}\mathcal{X}$$ is a weak equivalence.
\end{definition}
\begin{el}
(1) Let $\mathcal{X}:J \rightarrow \mathcal{C}$. 
Define $R\mathcal{X}:J_{+} \rightarrow \mathcal{C}$
as $$R\mathcal{X}(j)=
\begin{cases}
cholim_{J}\mathcal{X}, & \text{if }  j=\emptyset\\
\hat{F}\mathcal{X}(j), & \text{otherwise}
\end{cases} $$ 
Then $R\mathcal{X}$ is homotopy Cartesian.

(2) Let $F:\mathcal{C}\rightarrow \mathcal{D}$
be a functor between simplicial
model categories and $J$ a small category.
Suppose that $F$ preserves weak equivalences and homotopy 
Cartesian objects $J_{+}\rightarrow \mathcal{C}$.
Then for any $\mathcal{X}:J\rightarrow \mathcal{C}$,
$F(cholim_{J}\mathcal{X})$ and $cholim_{J}F\mathcal{X}$ are
weakly equivalent.

(3) Let $\mathcal{X},\mathcal{Y}:J_{+}\rightarrow  
\mathcal{C}$ be two functors.
If $\mathcal{X}\rightarrow \mathcal{Y}$ is an objectwise weak 
equivalence, then $\mathcal{X}$ is homotopy Cartesian if and 
only if $\mathcal{Y}$ is.

(4) Suppose that $J$ a small Reedy category with cofibrant constants. 
Let $\mathcal{X}\in \mathcal{C}^{J_{+}}$ be a fibrant object. 
Then \cite[19.9.1(2)]{Hi} the natural map $lim_{ J}\mathcal{X}
\rightarrow holim_{ J}\mathcal{X}$ is a weak equivalence, so 
in this case $\mathcal{X}$ is homotopy Cartesian if and only if 
$\mathcal{X}(\emptyset)\rightarrow lim_{J}\mathcal{X}$ is a 
weak equivalence.
\end{el}
\begin{repl} \cite[Remark 1.11]{Go2}
Let $J$ be a small Reedy category with cofibrant constants
and such that every map in the direct subcategory 
$\overrightarrow{J}$ is a monomorphism. The natural
functor $J\rightarrow J_{+}$ is an inclusion, let's call it $u$.
The functor $u^{\ast}:\mathcal{C}^{J_{+}}\rightarrow 
\mathcal{C}^{J}$ has a (full and faithful) right adjoint $u_{\ast}$ 
calculated as 
$$u_{\ast}\mathcal{X}(j)=
\begin{cases}
lim_{J}\mathcal{X}, & \text{if }  j=\emptyset\\
\mathcal{X}(j), & \text{otherwise}
\end{cases} $$ 
Let $\mathcal{X}:J\rightarrow \mathcal{C}$ be a fibrant object.
For each $j\in J$ we denote by $Q_{j}$ the composite $(j\downarrow J)
\rightarrow J\overset{\mathcal{X}}\rightarrow \mathcal{C}$.
Define $\mathcal{Z}:J_{+}\rightarrow \mathcal{C}$
as
$$\mathcal{Z}(j)=
\begin{cases}
holim_{J}\mathcal{X}, & \text{if }  j=\emptyset\\
holim_{(j\downarrow J)}Q_{j}, & \text{otherwise}
\end{cases} $$ 
Then $\mathcal{Z}$ is homotopy Cartesian and 
the natural map $\mathcal{X}\rightarrow u^{\ast}\mathcal{Z}$
is an objectwise weak equivalence.
\end{repl}
Let $\mathcal{C}$ be a simplicial model category,
$J$ a small category and 
$\mathcal{X}:(\underset{pb}\int J)_{+}
\rightarrow \mathcal{C}$. Using the isomorphism
(3) from Example 1.2 we can identify $\mathcal{X}$
with an arrow $\mathcal{X}_{left}\rightarrow 
\mathcal{X}_{right}$ of $\mathcal{C}^{J_{+}}$. 
If $J$ a Reedy category with cofibrant 
constants we have a homotopy pullback diagram
\[
\xymatrix{
cholim_{\underset{pb}\int J}\mathcal{X} \ar[r] \ar[d] & 
\hat{F}\mathcal{X}_{right}(\emptyset) \ar[d]\\
cholim_{ J}\mathcal{X}_{left} \ar[r] & 
cholim_{J}\mathcal{X}_{right}\\
}
  \]
\begin{lem} \cite[Proposition 1.6]{Go2}
Let $\mathcal{C}$ be a simplicial model category and $J$ a 
small Reedy category with cofibrant constants. Let 
$\mathcal{X}:(\underset{pb}\int J)_{+} \rightarrow \mathcal{C}$ 
be a functor.

$(i)$ If $\mathcal{X}_{left}$ and $\mathcal{X}_{right}$ are 
homotopy Cartesian then so is $\mathcal{X}$.

$(ii)$ If $\mathcal{X}$ and $\mathcal{X}_{right}$ are 
homotopy Cartesian then so is $\mathcal{X}_{left}$.
\end{lem}
Let $\mathcal{C}$ be a pointed simplicial model category 
and $J$ a small category. We denote by $\partial$ the composite
functor $$\mathcal{C}^{(\underset{pb}\int J)_{+}}\cong
(\mathcal{C}^{ J_{+}})^{[1]}\overset{{\rm chf}}\longrightarrow 
\mathcal{C}^{ J_{+}}$$ so that 
$\partial\mathcal{X}(\emptyset)={\rm chf}(\mathcal{X}_{left}
(\emptyset)\rightarrow \mathcal{X}_{right}(\emptyset))$ and
$\partial\mathcal{X}(j)={\rm chf}(\mathcal{X}_{left}(j)
\rightarrow \mathcal{X}_{right}(j))$. It follows that
$$holim_{J}\partial\mathcal{X}\cong {\rm hf}(cholim_{J}
\mathcal{X}_{left}\rightarrow cholim_{J}\mathcal{X}_{right})$$
If $\mathcal{X},\mathcal{Y}:(\underset{pb}\int J)_{+}\rightarrow  
\mathcal{C}$ are two functors and $\mathcal{X}\rightarrow \mathcal{Y}$ 
is an objectwise weak equivalence, then 
$\partial\mathcal{X}\rightarrow \partial\mathcal{Y}$ 
is an objectwise weak equivalence.
\begin{lem} \cite[Lemma 4.7]{Ku}
Let $\mathcal{C}$ be a pointed simplicial model category and 
$J$ a small Reedy category with cofibrant constants. Let 
$\mathcal{X}:(\underset{pb}\int J)_{+}\rightarrow \mathcal{C}$ 
be a functor. If $\mathcal{X}$ is homotopy Cartesian then so is 
$\partial\mathcal{X}$. If $\mathcal{C}$ is moreover 
a stable model category, then the converse holds.
\end{lem}
\begin{lem}
Let $J$ be a small category and 
$F:\mathcal{C}\rightarrow \mathcal{D}$ a
functor between pointed simplicial model 
categories such that 

$(i)$ $F0\rightarrow 0$ is a weak equivalence, and

$(ii)$ $F$ sends homotopy pullback 
diagrams of the form 
\[
   \xymatrix{
A \ar[r] \ar[d] & 0 \ar[d]\\
B \ar[r] & C\\
}
  \]
to homotopy pullbacks.

Let $\mathcal{X}:(\underset{pb}\int J)_{+}\rightarrow 
\mathcal{C}$ be a functor. Then there are a functor 
$H: J_{+}\rightarrow \mathcal{D}$
and a diagram of functors 
$F\partial \mathcal{X}\rightarrow H \leftarrow 
\partial F\mathcal{X}$ in which every map is an 
objectwise weak equivalence.
\end{lem}
\begin{proof}
This follows from the considerations in 2.1.
\end{proof}
\begin{corollary} \cite[Proof of Lemma 1.19]{Ch}
Let $J$ be a small Reedy category with cofibrant 
constants and $\mathcal{C}$ and $\mathcal{D}$ pointed 
simplicial model categories with $\mathcal{D}$ stable.
Let $F:\mathcal{C}\rightarrow \mathcal{D}$ be
a functor such that 

$(i)$ $F0\rightarrow 0$ is a weak equivalence, and

$(ii)$ $F$ sends homotopy pullback 
diagrams of the form 
\[
   \xymatrix{
A \ar[r] \ar[d] & 0 \ar[d]\\
B \ar[r] & C\\
}
  \]
to homotopy pullbacks.

If $F$ preserves homotopy Cartesian objects 
$J_{+} \rightarrow \mathcal{C}$, then $F$ 
preserves homotopy Cartesian objects
$(\underset{pb}\int J)_{+} \rightarrow \mathcal{C}$.
\end{corollary}
\begin{proof}
One uses Lemmas 6.5 and 6.6, and 6.2(3). 
\end{proof}
\begin{lem}
Let $\mathcal{C}$ be a locally finitely presentable simplicial 
model category whose tensor, viewed as a functor
${\bf S}\times \mathcal{C}\rightarrow \mathcal{C}$, 
preserves finitely presentable objects, and such that 
an $\omega$-indexed colimit of weak equivalences of 
$\mathcal{C}$ is a weak equivalence and an 
$\omega$-indexed colimit of fibrant objects of $\mathcal{C}$ 
is fibrant. Let $J$ be a finite category such that for each 
object $j$ of $J$, the nerve of $(J\downarrow j)$ is finitely 
presentable. Then a sequential corrected homotopy 
colimit of homotopy Cartesian objects of 
$\mathcal{C}^{J_{+}}$ is homotopy Cartesian.
\end{lem}
\begin{proof}
This is a consequence of Corollary 5.5.
\end{proof}

\section{On (strongly) homotopy co-Cartesian cubes}
In this section we give a model theoretic interpretation 
of the notions of homotopy co-Cartesian and strongly homotopy 
co-Cartesian cube from \cite[Definitions 1.4 and 2.1]{Go2}.
\\

Let $\mathcal{C}$ be a fixed model category. Let $J$ be
a small Reedy category. Following \cite[Chapter 15]{Hi},
we denote by $F^{n}J$ the $n$-filtration of $J$ and by
$I^{n}:F^{n-1}J\rightarrow F^{n}J$ the inclusion functor.
An object of $\mathcal{C}^{J}$ is cofibrant if and 
only if for every $n\geq 0$, its restriction to $F^{n}J$
is cofibrant. For each object $\alpha$ of $J$ of degree $n$, 
$(I^{n}\downarrow \alpha)$ is a Reedy category.
The restriction functor $I^{n \ast}:\mathcal{C}^{F^{n}J}
\rightarrow\mathcal{C}^{F^{n-1}J}$ has a left adjoint 
$I^{n}_{!}$ and a right adjoint $I^{n}_{\ast}$. 
\begin{lem}
Suppose that $J$ has fibrant constants. Then both
$F^{n}J$ and $(I^{n}\downarrow \alpha)$ have fibrant 
constants, where $\alpha$ is an object of $J$ of degree 
$n$, and $(I^{n}_{!},I^{n \ast})$ is a Quillen pair.
\end{lem}
The functor $I^{n \ast}$ is a cloven Grothendieck 
bifibration. The fiber category of $I^{n \ast}$ over 
$\mathcal{X}\in \mathcal{C}^{F^{n-1}J}$ is denoted by
$(\mathcal{C}^{F^{n}J})_{\mathcal{X}}$.
A cartesian lift of $u:\mathcal{Y}\rightarrow 
I^{n \ast}\mathcal{Z}$ is given by the pullback diagram
\[
\xymatrix{
u^{\ast}\mathcal{Z} \ar[r] \ar[d] & \mathcal{Z} \ar[d]\\
I^{n}_{\ast}\mathcal{Y} \ar[r] & I^{n}_{\ast}I^{n \ast}
\mathcal{Z}\\
}
  \]
A cocartesian lift of $u:I^{n \ast}\mathcal{Z}\rightarrow 
\mathcal{Y}$ is given by the pushout diagram
\[
\xymatrix{
I^{n}_{!}I^{n \ast}\mathcal{Z}\ar[r] \ar[d] & I^{n}_{!}
\mathcal{Y} \ar[d]\\
\mathcal{Z} \ar[r] &  u_{!}\mathcal{Z}\\
}
\]
Every map $f:\mathcal{X}\rightarrow \mathcal{Y}$ of
$\mathcal{C}^{F^{n}J}$ can be decomposed as
$\mathcal{X}\overset{f^{u}}\longrightarrow 
u^{*}(\mathcal{Y})\overset{cart}\longrightarrow 
\mathcal{Y}$, where $u=p(f)$ and $cart:u^{*}(\mathcal{Y})
\rightarrow \mathcal{Y}$ is cartesian over $u$, and as
$\mathcal{X}\overset{cocart}\longrightarrow 
u_{!}(\mathcal{X})\overset{f_{u}}\rightarrow 
\mathcal{Y}$, where $cocart:\mathcal{X}\rightarrow 
u_{!}\mathcal{X}$ is cocartesian over $u$.
\begin{lem}
For every object $\mathcal{X}$ of $\mathcal{C}^{F^{n-1}J}$, 
$(\mathcal{C}^{F^{n}J})_{\mathcal{X}}$ has a 
model structure in which a map $f$ is a weak equivalence, 
cofibration or fibration if $f$ is so in $\mathcal{C}^{F^{n}J}$.
\end{lem}
\begin{proof}
The factorization axiom of a model category is proved
inductively on the degree of the objects of $\mathcal{C}$, 
exactly as in \cite[15.3.16]{Hi}. The difference with 
\emph{loc. cit.} is in degrees $\leq n-1$, when we choose 
the factorization to be given by identity maps.
\end{proof}
\begin{theorem}
Let $J$ be a small Reedy category with fibrant constants.
For each $n\geq 1$ the category $\mathcal{C}^{F^{n}J}$
admits a model structure in which a map  
$f:\mathcal{X}\rightarrow \mathcal{Y}$ is a 

$\bullet$ weak equivalence if $I^{n \ast}(f)$ is a weak 
equivalence in $\mathcal{C}^{F^{n-1}J}$;

$\bullet$ cofibration if $f_{u}$ is a trivial cofibration in 
$(\mathcal{C}^{F^{n}J})_{I^{n \ast}\mathcal{Y}}$
 and $I^{n \ast}(f)$ is a cofibration in
$\mathcal{C}^{F^{n-1}J}$;

$\bullet$ fibration if $f^{u}$ is a fibration in 
$(\mathcal{C}^{F^{n}J})_{I^{n \ast}\mathcal{X}}$ and 
$I^{n \ast}(f)$ is a fibration in $\mathcal{C}^{F^{n-1}J}$.

We denote this model structure by $L_{r}\mathcal{C}^{F^{n}J}$.
$L_{r}\mathcal{C}^{F^{n}J}$ is a right Bousfield localization
of $\mathcal{C}^{F^{n}J}$. The adjoint pair
$$I^{n}_{!}:\mathcal{C}^{F^{n-1}J}
\rightleftarrows L_{r}\mathcal{C}^{F^{n}J}:I^{n \ast}$$
is a Quillen equivalence.
\end{theorem}
\begin{proof}
This is an application of Theorem 7.4, using Lemmas 7.1 and 7.2.
\end{proof}
The next result is a consequence of \cite[2.2]{St}. 
\begin{theorem}
Let $p:\mathbb{E}\rightarrow \mathbb{B}$ be a
cloven bifibration. Suppose that

$(i)$ the base category $\mathbb{B}$ has a model structure
$(\mathfrak{Cof},\mathfrak{W},\mathfrak{Fib})$;

$(ii)$ for each object $I$ of $\mathbb{B}$, the fibre category
$\mathbb{E}_{I}$ has a model structure $(\mathfrak{Cof}_{I},
\mathfrak{W}_{I},\mathfrak{Fib}_{I})$;

$(iii)$ for every morphism $u:I\rightarrow J$ of $\mathbb{B}$, we
have $u^*(\mathfrak{Fib}_{J})\subseteq \mathfrak{Fib}_{I}$
and $u^*(\mathfrak{Fib}_{J}\cap \mathfrak{W}_{J})\subseteq 
\mathfrak{Fib}_{I}\cap \mathfrak{W}_{I}$.

Then $\mathbb{E}$ has a model structure in which a map 
$f$ is a

$\bullet$ weak equivalence if $p(f)$ is a weak equivalence in 
$\mathbb{B}$;

$\bullet$ cofibration if $f_{u} \in \mathfrak{Cof}_{p(cod(f))}
\cap\mathfrak{W}_{p(cod(f))}$ and $p(f)\in \mathfrak{Cof}$;

$\bullet$ fibration if $f^{u} \in \mathfrak{Fib}_{p(dom(f))}$ 
and $p(f)\in \mathfrak{Fib}$.
\end{theorem}
We spell out that an object $\mathcal{X}$ of $\mathcal{C}^{F^{n}J}$
is cofibrant in $L_{r}\mathcal{C}^{F^{n}J}$ if and only if 
$I^{n \ast}\mathcal{X}$ is cofibrant in $\mathcal{C}^{F^{n-1}J}$ 
and $I^{n}_{!}I^{n \ast}\mathcal{X}\rightarrow \mathcal{X}$
is a trivial cofibration in $\mathcal{C}^{F^{n}J}$ if and only if
$I^{n \ast}\mathcal{X}$ is cofibrant in $\mathcal{C}^{F^{n-1}J}$ 
and for each object $\alpha$ of $F^{n}J$ of degree $n$, the
canonical map $colim_{(I^{n}\downarrow \alpha)}I^{n \ast}\mathcal{X}
\rightarrow \mathcal{X}(\alpha)$ is a trivial cofibration. By 
\cite[15.2.9]{Hi} this is equivalent to saying that $\mathcal{X}$ 
is cofibrant in $\mathcal{C}^{F^{n}J}$ and for each object 
$\alpha$ of $F^{n}J$ of degree $n$, the latching map 
$L_{\alpha}\mathcal{X}\rightarrow \mathcal{X}(\alpha)$
of $\mathcal{X}$ at $\alpha$ is a weak equivalence. 

Suppose now that $\mathcal{C}$ is a simplicial model category. Let
$\mathcal{X}:J\rightarrow \mathcal{C}$. For each object $\alpha$
of $J$ of degree $n$ we have a commutative diagram
\[
\xymatrix{
(c)hocolim_{\partial(\overrightarrow{J}\downarrow \alpha)}\mathcal{X} 
\ar[d] \ar[r] & colim_{\partial(\overrightarrow{J}\downarrow \alpha)}
\mathcal{X}=L_{\alpha}\mathcal{X} \ar[d]^{\cong} \ar[r] & \mathcal{X}
(\alpha)\\
(c)hocolim_{(I^{n}\downarrow \alpha)}\mathcal{X} 
\ar[r] &  colim_{(I^{n}\downarrow \alpha)}\mathcal{X}\\
}
  \]
\begin{example}
Let $n\geq 2$. The category $\mathcal{P}(\underline{n})$ (Example 1.4) 
becomes a Reedy category with fibrant constants if we let the 
direct subcategory $\overrightarrow{\mathcal{P}(\underline{n})}$ 
be $\mathcal{P}(\underline{n})$, the inverse subcategory 
$\overleftarrow{\mathcal{P}(\underline{n})}$
be the discrete category on the set of objects of 
$\mathcal{P}(\underline{n})$ and the degree of 
$S\in \mathcal{P}(\underline{n})$ be $|S|$.
One has $F^{0}\mathcal{P}(\underline{n})=\ast$,
$F^{1}\mathcal{P}(\underline{n})=\underline{n}_{+}$,
$F^{n-1}\mathcal{P}(\underline{n})=
\mathcal{P}_{1}(\underline{n})$ and $F^{n}\mathcal{P}
(\underline{n})=\mathcal{P}(\underline{n})$.
If $S\subset \underline{n}$ has cardinality $k$,
$\partial(\overrightarrow{\mathcal{P}
(\underline{n})}\downarrow S)=(I^{k}\downarrow S)
\cong \mathcal{P}_{1}(S)$. The cofibrant 
objects of $\mathcal{C}^{\mathcal{P}(\underline{n})}$ 
are the cofibration cubes of \cite[Definition 1.13]{Go2}.
The cofibrant objects of $L_{r}\mathcal{C}^{\mathcal{P}
(\underline{n})}$ are the cofibration cubes $\mathcal{X}$ for 
which $colim_{\mathcal{P}_{1}(\underline{n})}\mathcal{X}
\rightarrow \mathcal{X}(\underline{n})$ is a weak equivalence.
In general, for $k\geq 2$, the cofibrant objects of $L_{r}\mathcal{C}^{F^{k}
\mathcal{P}(\underline{n})}$ are the cofibrant objects $\mathcal{X}$ 
for which $colim_{\mathcal{P}_{1}(S)}\mathcal{X}\rightarrow \mathcal{X}(S)$ 
is a weak equivalence for each subset $S$ of $\underline{n}$ 
with $|S|=k$. If $\mathcal{C}$ is moreover a simplicial model 
category, by \cite[19.9.1(1)]{Hi} the cofibrant objects of 
$L_{r}\mathcal{C}^{\mathcal{P}(\underline{n})}$ are the 
cofibration cubes $\mathcal{X}$ for which 
$hocolim_{\mathcal{P}_{1}(\underline{n})}\mathcal{X}
\rightarrow \mathcal{X}(\underline{n})$ is a weak equivalence,
and for $k\geq 2$, the cofibrant objects of $L_{r}\mathcal{C}^{F^{k}
\mathcal{P}(\underline{n})}$ are the cofibrant objects $\mathcal{X}$ 
for which $hocolim_{\mathcal{P}_{1}(S)}\mathcal{X}\rightarrow 
\mathcal{X}(S)$ is a weak equivalence for each subset $S$ of 
$\underline{n}$ with $|S|=k$. 
Using the dual of Lemma 6.4 and properties of homotopy pushouts,
it follows that a cofibrant object $\mathcal{X}$ of 
$\mathcal{C}^{\mathcal{P}(\underline{n})}$ 
is cofibrant in $L_{r}\mathcal{C}^{F^{k}\mathcal{P}(\underline{n})}$
for each $2\leq k\leq n$ if and only if $\mathcal{X}$ is strongly 
homotopy co-Cartesian in the sense of \cite[Definition 2.1]{Go2}.
\end{example}

\end{document}